\documentclass[preprint, 12pt]{elsarticle}
\usepackage{amssymb, amsthm, amsmath}

\theoremstyle{plain}
\newtheorem{theorem}{Theorem}
\newtheorem{corollary}[theorem]{Corollary}
\newtheorem{lemma}[theorem]{Lemma}

\newtheorem{conj}[theorem]{Conjecture}

\theoremstyle{definition}

\newtheorem{example}[theorem]{Example}
\newtheorem{remark}[theorem]{Remark}

\numberwithin{theorem}{section}

\newcommand{\F}{{\mathbb F}}

\newcommand{\Q}{{\mathbb Q}}

\newcommand{\Z}{{\mathbb Z}}
\renewcommand{\int}{\operatorname{int}}

\renewcommand{\wp}{{\mathfrak p}}

\newcommand{\en}[1]{{\rm End}(#1)}

\newcommand{\He}[2]{\mathcal{H}_{#1,#2}}
\newcommand{\HT}[2]{\widetilde{\mathcal{H}}_{#1,#2}}
\newcommand{\Hass}[2]{\mathcal{H}_{#1,#2}^{ss}}
\newcommand{\HassT}[2]{\widetilde{\mathcal{H}}_{#1,#2}^{ss}}

\newcommand{\TT}{\mathcal{T}}
\newcommand{\GL}{\mathbf{GL}}
\newcommand{\SL}{\mathbf{SL}}
\newcommand{\owp}{\widetilde{\wp}}

\newcommand{\DMF}[2]{M_{#1, #2}}
\newcommand{\DMFmodT}[2]{\widetilde{M}_{#1, #2}}
\newcommand{\ov}[1]{\widetilde{#1}}

\journal{arXiv.org}

\begin{document}

\begin{frontmatter}


\title{On the action of Hecke operators on Drinfeld modular forms}

\author[kj]{Kirti Joshi} 

\author[ap]{Aleksandar Petrov}

\address[kj]{
Department of Mathematics, University of Arizona, 617 N Santa Rita, Tucson
85721-0089, USA.}

\address[ap]{ Science Program, Texas A \& M University at Qatar, Doha 23874, Qatar.}


\begin{abstract}In this paper we determine the explicit structure of the semisimple part of the Hecke algebra that acts on Drinfeld modular forms of full level modulo~$T$. We show that modulo $T$ the Hecke algebra has a non-zero semisimple part. In contrast, a well-known theorem of Serre asserts that for classical modular forms the action of $\TT_\ell$ for any odd prime $\ell$ is nilpotent modulo $2$. After proving the result for Drinfeld modular forms modulo $T$, we use computations of the Hecke action modulo~$T$ to show that certain powers of the Drinfeld modular form $h$ cannot be eigenforms. Finally, we conjecture that the Hecke algebra that acts on Drinfeld modular forms of full level is not smooth for large weight, which again contrasts the classical situation. 
\end{abstract}


\begin{keyword} Drinfeld modular forms \sep Reduction of Drinfeld modular forms modulo $T$

\end{keyword}

\end{frontmatter}


\section{Introduction}

 The nilpotence of the Hecke algebra modulo $2$ for the full modular group $\SL_2(\Z)$ was first observed by Serre (\cite[\S 5]{Serre_Valeurs}). The result has many important arithmetic applications (see Sec.~2.7 and Ch.~5 of \cite{ono2004web}). Recently, Nicolas and Serre (\cite{SerreNicolas_nilpotence}, \cite{SerreNicolas_Hecke_algebra}) have determined the order of nilpotence and the structure of the Hecke algebra modulo $2$ for the full modular group. Ono and Taguchi (\cite{Ono_Taguchi}) have studied other spaces of modular forms for which the Hecke action is nilpotent modulo $2$ and have given consequences of the nilpotence for quadratic forms, partition functions and central values of twisted modular $L$-functions. The main goal of the present work is to study the structure of the Hecke algebra that acts on Drinfeld modular forms for $\GL_2(\F_q[T])$ and their reductions modulo $T$. Unlike the case of classical modular forms, the Hecke eigenvalues are not directly related to the `Fourier' coefficients of Drinfeld modular forms. This makes the analysis of the Hecke algebra even more difficult and the properties of the Hecke algebra remain somewhat of a mystery. Our approach uses Serre derivatives of Drinfeld modular forms. We completely describe the semisimple part of the action of $\TT_\wp \mod{T}$, where $T \neq \wp \in \F_q[T]$, $\wp$ irreducible.  This is done in Theorem \ref{hecke_algebra_mod_T}. The algebra of Drinfeld modular forms modulo any degree one prime is freely generated over $\F_q$ by the reduction of the normalized cuspidal Drinfeld modular form of the smallest weight (see \cite[(12.9)]{Gek}), which parallels the classical case where the algebra of modular forms modulo $2$ is freely generated over $\F_2$ by the reduction of the classical delta function. This suggests that degree one primes play a role analogous to the prime $2$. In contrast with the result of Serre, we show that $\TT_\wp \mod{T}$ has a non-trivial semisimple part for all $\wp \neq T$. The result also holds with appropriate modifications for any degree one prime (see Remark~\ref{other_primes} below).
 
 Because the Hecke action modulo $T$ is a reduction of the Hecke action on Drinfeld modular forms, explicit calculations of the Hecke action modulo~$T$ allow us to prove that certain powers of $h$ (the normalized Drinfeld cuspidal form of smallest weight) cannot be eigenforms in~Section~\ref{no_reduction_mod_T}. Classically powers of eigenforms are never eigenforms, but since the `Fourier' coefficients of Drinfeld eigenforms are not directly related with the eigenvalues this is no longer the case for Drinfeld modular forms and it is interesting to determine which powers of a given eigenform remain eigenforms.

 Finally, we conjecture (see Conjecture~\ref{main_conjecture})  that the Hecke algebra is not smooth over the rational function field. Classically the Hecke algebra that acts on the space of modular forms of full level is \'{e}tale (and therefore smooth) over $\Q$, therefore our conjecture underscores another difference between the structure of the Hecke algebra in the case of classical modular forms and in the case of Drinfeld modular forms.

\section{Notation}

Let $p$ be a prime number. Let $q = p^r$, with $r$ a positive integer. Let $A$ be the polynomial ring $\F_q[T]$, $A_+$ be the set of monic polynomials in $A$, $K$ the fraction field of $A$. Let $\phi$ denote the Carlitz module, i.e., the Drinfeld module of rank one, uniquely defined by requiring $\phi_T(z) = T z + z^q$. Throughout the article $\tilde{\pi}$ will denote a fixed choice of a period of the Carlitz module. The reader can find more about the Carlitz theory in \cite[Ch.2]{Tha}.

Let $\DMF{k}{m}$ denote the set of Drinfeld modular forms of weight $k$, type $m$, for $\GL_2 (A)$ over $K$ (see \cite[\S 5]{Gek} or \cite{Goss_mod_forms}). Let $t := t(z) = e_{\tilde{\pi} A} (\tilde{\pi} z)^{-1}$ be the usual uniformizer at `infinity', and let $t_a := t(az)$. The function $t_a$, for $a \in A_+$, can be expanded into a $t$-series with coefficients in $A$ (see \cite[(7.3)]{Gek}):
\begin{equation} \label{formula_for_ta}
 t_a = \frac{t^{q^{\deg_T {(a)}}}}{1 + \cdots} = t^{q^{\deg_T{(a)}}} + \text{higher order terms in $t$} \in A[[t]].
\end{equation}
Every element $f$ of $M_{k, m}$ has a $t$-expansion:
\[
	f = \sum_{n = 0}^\infty a_n t^n , \qquad \qquad a_n \in K,
\] which is well-defined for $z$ close to `infinity' and such an expansion determines $f$ uniquely. Let
\[
	g := 1 - (T^q - T) \sum_{a \in A_+} t_a^{q-1}, \qquad \qquad h := \sum_{a \in A_+} a^q t_a.
\]
We have defined $g$ and $h$ not by giving their $t$-expansions, but by giving their $A$-expansions (see \cite{Pet} for details). Most Drinfeld modular forms do not appear to have $A$-expansions, but the ones that do have the advantage that one can compute the Hecke action on them in an easy manner (for examples see \eqref{g_and_h_are_eigenforms} below).  

The algebra $\DMF{k}{m}$ is generated by $g$ and $h$  (\cite[(6.5) \& (9.1)]{Gek}). More precisely, 
\begin{equation} \label{basis_for_M_k_m}
	\DMF{k}{m} = K \langle g^i h^j : k = (q-1) i + (q + 1) j, \  j \equiv m \mod{(q-1)} \rangle.
\end{equation} 

Let
\[
	M := \bigoplus_{k, m} \DMF{k}{m}.
\]
Then
\[
	M = K [g, h].
\]
Let $M_T$ denote the subalgebra of $M$ which consists of Drinfeld modular forms that have $t$-expansion coefficients that have denominators relatively prime to $T$. Such forms can be reduced modulo~$T$ by reducing their coefficients modulo~$T$.

 Let
\[
	\DMFmodT{k}{m} := \{ \widetilde{f} \in \F_q[[t]] : \exists f \in M_{k, m} \cap M_T  \text{ s.t. } f \equiv \widetilde{f} \mod{T} \}.
\]  
We note that $\widetilde{g} = 1$, while $\widetilde{h} = t + \text{higher order terms in $t$}$.

If we set
\[
	\mu :=  \left \lfloor \frac{k - m(q+1)}{q^2 - 1} \right \rfloor,
\]
then \eqref{basis_for_M_k_m} and $\widetilde{g} = 1$ show that
\begin{equation} \label{basis_for_M_k_m_mod_T}
	\DMFmodT{k}{m} = \F_q \left \langle \ov{h}^m, \ov{h}^{m + (q-1)}, \ldots, \ov{h}^{m + \mu(q-1)} \right \rangle.
\end{equation}

Let
\[
	\widetilde{M} := \bigoplus_{k, m} \DMFmodT{k}{m}.
\]
We see that $\widetilde{M} = \F_q [\ov{h}]$.

The `false' Eisenstein series of Gekeler (\cite[(8.2)]{Gek}) is defined by
\[
	E := \sum_{a \in A_+} a t_a.
\]
The function $E$ is not a Drinfeld modular form, but it is quasi-modular and is closely related with differential properties of Drinfeld modular forms. 

From the definitions of $E$ and $h$ it is immediate that 
\begin{equation} \label{E_equals_h_mod_T}
E \equiv h \mod{T}, 
\end{equation}
and therefore	$\widetilde{M} = \F_q  [\ov{E}  ]$.

Let $\wp$ be a prime ideal of $A$ different from $(T)$. By abuse of notation, we will also denote by $\wp$ the monic irreducible that generates the ideal $\wp$. Let $f \in \DMF{k}{m}$. By definition (\cite[\S~1.8]{Goss_mod_forms}) the $\wp^\text{th}$ Hecke operator of weight $k$, $\TT_{\wp, k}$, acts on $f$ according to
\begin{equation} \label{definition_of_the_Hecke_action}
	\TT_{\wp, k} f := \wp^k f (\wp z) + \sum_{\beta \in S_\wp} f \left ( \frac{z + \beta}{\wp} \right),
\end{equation}
where $S_\wp := \{ \beta \in A: \deg_T(\beta) < \deg_T (\wp) \}$. 

The corresponding action of $\TT_{\wp, k}$ on a $t$-expansion is (see \cite[(7.3)]{Gek}):
\begin{equation} \label{action_on_t_exp}
 \TT_{\wp, k} \left ( \sum_{n = 0}^\infty a_n t^n \right) = \wp^k \sum_{n = 0}^\infty a_n t_\wp^n + \sum_{n=0}^\infty a_n G_{n, \wp} (\wp t),
\end{equation}
where $G_{n, \wp} (X)$ is the $n$-th Goss polynomial of the finite lattice formed by the $\wp$-torsion of the Carlitz module. The formulas are also valid when $\wp = T$. 

Since the coefficients of $G_{n, \wp} (X)$ are in $K$ this shows that
\[
	\TT_{\wp, k} \DMF{k}{m} \subset \DMF{k}{m}.
\]

We want to define the Hecke action on $\DMFmodT{k}{m}$ for primes different from $(T)$ by formula \eqref{action_on_t_exp}. We need to check that \eqref{action_on_t_exp} does not introduce denominators divisible by~$T$, so that reduction is well-defined. As we have already remarked, the $t$-expansion of $t_\wp$ has coefficients in $A$, therefore it can be reduced modulo~$T$. It remains to check that for $\wp \neq T$, $G_{n, \wp}(X)$ has coefficients with denominators relatively prime to $T$. To that end, we note that the normalized exponential function of the lattice of $\wp$-torsion of the Carlitz module, $e_\wp(z)$, and the Carlitz module at $\wp$, $\phi_\wp(z)$, have the same zeroes. Since $e_\wp(z) = z + \cdots$, while $\phi_\wp(z) = \wp z + \cdots$, we have
\[
	e_\wp(z) = \frac{1}{\wp} \phi_\wp (z).
\]
Because $\phi_\wp(z)$ has coefficients in $A$, we conclude that $e_\wp(z)$ has coefficients with denominators that are relatively prime to $T$, provided that $\wp \neq T$. According to \cite[(3.3)]{Gek} the polynomials $G_{n, \wp}(X)$ can be inductively computed from the coefficients of $e_\wp(z)$ without introducing denominators. Therefore $G_{n, \wp}(X)$ has coefficients that have no powers of $T$ in its denominators and can be reduced modulo $T$.

We conclude that if $\wp \neq T$, then
\[
	\TT_{\wp, k} (\DMF{k}{m} \cap M_T) \subset \DMF{k}{m} \cap M_T,
\]
where $\TT_{\wp, k}$ is defined by formula \eqref{action_on_t_exp}.

For future reference we give explicit formula for $e_\wp(z)$ and $G_{n+1, \wp} (X)$ based on the coefficients of $\wp$. Let $\wp = \alpha_0 + \alpha_1 T + \cdots + T^d$, $\alpha_i \in \F_q$. Since $\phi_T (z) \equiv z^q \mod{T}$, it follows that
\[
	\phi_\wp(z) \equiv \sum_{n = 0}^d \alpha_n z^{q^n} \mod{T}.
\]
Hence
\begin{equation} \label{exp_mod_T}
        e_\wp(z) \equiv  \frac{1}{\owp} \sum_{n = 0}^d \alpha_n z^{q^n} \equiv \sum_{n =0}^{d} \frac{\alpha_n}{\alpha_0} z^{q^n} \mod{T}.
\end{equation}
Let $\underline{i} = (i_0, \ldots, i_s)$ run over the set of $(s+1)$-tuples ($s$ arbitrary) satisfying $i_0 + \cdots + i_s = m$, $i_0 + i_1 q + \cdots + i_s q^s = n$. According to \cite[(3.8)]{Gek}  formula \eqref{exp_mod_T} shows that
\begin{equation} \label{Goss_poly_mod_T}
	G_{n+1, \wp} (X) \equiv \sum_{m \leq n} \left ( \sum_{\underline{i}} \binom{m}{i_0, \ldots, i_s} \left ( \frac{\alpha_1}{\alpha_0} \right)^{i_1} \cdots \left ( \frac{\alpha_s}{\alpha_0} \right)^{i_s}  \right) X^{m+1} \mod{T},
\end{equation}
where $\binom{m}{i_0, \ldots, i_s}$ is the multinomial coefficient $m!/(i_0! \cdots i_s!)$. 

Corollary 3.1.13 from \cite{Pet} shows that for~$1 \leq j \leq q$:
\begin{equation} \label{g_and_h_are_eigenforms}
	\TT_{\wp, q-1} g = \wp^{q-1} g, \qquad \TT_{\wp, j (q+1)} h^j = \wp^j h^j.
\end{equation}

 Let $\He{k}{m}\subset\en{\DMF{k}{m}}$ be the subalgebra generated by all the Hecke operators over $K$ (see \cite[Df~3.1]{Goss_eisenstein} for the definition of $\TT_\mathfrak{m}$ when $\mathfrak{m}$ is not a prime ideal). The reader should note that just like the classical case $\TT_{\mathfrak{m}} \TT_{\mathfrak{n}}$ for relatively prime $\mathfrak{m}$ and $\mathfrak{n}$, but unlike the classical case $\TT_{\wp^N} = \TT_{\wp}^N$ for all $\wp$ and all positive integers $N$ (see~\cite[Pr~3.3]{Goss_eisenstein}). This allows us to focus our attention on $\TT_\wp$ where $\wp$ is prime. Let $\HT{k}{m}$ be the subalgebra of $\en{\DMFmodT{k}{m}}$ generated by the Hecke operators away from~$T$. Let $\Hass{k}{m}$ be the subalgebra of $\He{k}{m}$ generated by operators which are absolutely semisimple. Let $\HassT{k}{m}$ be the subalgebra of $\HT{k}{m}$ generated by absolutely semisimple elements. The algebra $\He{k}{m}$ is a finite-dimensional $K$-algebra while $\HT{k}{m}$ is a finite-dimensional $\F_q$-algebra. We note that $\Hass{k}{m}$ is a product of finite separable extensions (see \cite[Chap V, page 34]{BourbakiAlgebra}). We are interested in properties of $\Hass{k}{m}$ and $\HassT{k}{m}$.

\section{Hecke Algebra modulo $T$}

In this section we determine the structure of $\HassT{k}{m}$ completely (Theorem~\ref{hecke_algebra_mod_T} below). We will make use of the Serre derivatives from \cite[(8.5)]{Gek}, which we now recall. 

Define
\[
	\Theta := \frac{1}{\tilde{\pi}} \frac{d}{dz} = -t^2 \frac{d}{dt}.
\]
The operator $\Theta$ does not preserve $M$, but the $k^\text{th}$ Serre derivative
\[
	\partial_k := \Theta + k E
\]
does. One can show (see \cite[(8.5)]{Gek}) that $\partial_k : M_{k, m} \to M_{k+2, m+1}$ and that for $k = k_1 + k_2$, $f_1, f_2$ Drinfeld modular forms of weights $k_1, k_2$, respectively, we get $\partial_k (f_1 f_2) = \partial_{k_1} (f_1) f_2 + f_1 \partial_{k_2} (f_2)$. 

We have (paragraph following \cite[(8.7)]{Gek})
\[
	\partial_{q+1} h = 0,
\]
which, together with \eqref{E_equals_h_mod_T}, shows that $\Theta$ preserves $\widetilde{M}$. Indeed $\partial_{q+1} h = 0$ implies that
\begin{equation} \label{differentiate_h}
	\Theta h \equiv - h^2 \mod{T}.
\end{equation}

As in the previous section, let $\wp$ stand for either a prime ideal of $A$ different from~$(T)$ or for the monic generator of such an ideal. The reduction of $\wp$ modulo $T$, i.e., the constant term of $\wp$, will be denoted by $\owp$.
By using \eqref{definition_of_the_Hecke_action}, we get
\[
	\frac{d}{d z} \TT_{\wp, k} f = \wp^{k + 1} f'(\wp z) + \frac{1}{\wp} \sum_{\beta \in S_\wp} f' \left ( \frac{z + \beta}{\wp} \right),
\]
which in terms of operators shows that
\begin{equation} \label{differentiate_hecke}
 \wp \cdot \Theta \TT_{\wp, k} = \TT_{\wp, k + 2} \Theta.
\end{equation}

\begin{remark}
We will generally suppress the weight $k$ from the notation for the $\wp^{\text{th}}$ Hecke operator, since the weight can be read off from the form that the Hecke operator is acting on. We want to remark that modulo $T$ equation \eqref{differentiate_hecke} becomes
\begin{equation} \label{differentiate_hecke_mod_T}
\owp \cdot \Theta \TT_{\wp, k} \equiv \TT_{\wp, k + 2 + s(q-1)} \Theta \mod{T},
\end{equation}
where $s$ can be chosen to be any non-negative integer. 
\end{remark}

\begin{theorem} \label{main}
Let $\wp$ be a prime different from $T$. Let $n$ be a non-negative integer. Then
\begin{equation} \label{main_equation}
	\TT_\wp h^n \equiv \owp^n h^n + \text{lower order terms in $h$} \mod{T}.
\end{equation}
\end{theorem}

\begin{proof}
We prove that
\[
	\TT_\wp h^n \equiv \owp^n  h^n + \text{lower order terms in $h$} \mod{T}
\]
by downward induction on $n$. 

Let $n$ be given and let $p^\nu$ be the smallest power of $p$ bigger than or equal to $n$.  

If $0 \leq n \leq p$, then we already remarked in \eqref{g_and_h_are_eigenforms} that
\[	
	\TT_\wp g = \wp^{q-1} g, \qquad \qquad \TT_\wp h^n = \wp^n h^n \qquad (1 \leq n \leq p), 
\]
and therefore
\[
	\TT_\wp \widetilde{g} = \TT_\wp 1 = \owp^{q-1} = 1, \qquad \qquad \TT_\wp \widetilde{h^n} = \owp^n \widetilde{h^n} \qquad (1 \leq n \leq p).
\] 
This proves the result for $1 \leq n \leq p$.

Suppose that $\nu > 1$, i.e., $p^{\nu-1} < n \leq p^\nu$. Assume that the result is true for $n_0$ in the range $1 \leq n_0 \leq p^{\nu-1}$. 

If $p \mid n$, then $n = p n_0$ for $n_0$ between $p^{\nu-2}$ and $p^{\nu-1}$ and by the induction hypothesis
\[
	\TT_\wp h^n = (\TT_\wp h^{n_0})^p \equiv \left ( \owp^{n_0} h^{n_0} \right)^p + \text{lower order terms in $h$} \mod{T}.
\] 

If $p \nmid n$, then write
\[
	\TT_\wp h^n \equiv \epsilon_n h^n + \text{lower order terms in $h$} \mod{T}.
\]
We apply $\owp \Theta$ to both sides of this equation and use \eqref{differentiate_hecke_mod_T}:
\[
\begin{aligned}
	\TT_\wp \left (\Theta h^n \right ) & \equiv \owp \cdot \Theta \left (\TT_\wp h^n \right) \\
	& \equiv \owp \cdot  \Theta \left( \epsilon_n h^n + \text{lower order terms in $h$} \right ) \bmod T.
\end{aligned}
\]
Using \eqref{differentiate_h} this becomes
\[
	\TT_\wp ( - n h^{n+1} ) \equiv - n \owp \epsilon_n \cdot  h^{n+1} + \text{lower order terms in $h$} \mod{T}.
\]
As $p \nmid n$, we have
\[
	\TT_\wp (h^{n+1}) \equiv \owp \epsilon_n \cdot h^{n + 1} + \text{lower order terms in $h$} \mod{T}.
\]
The equation above shows that one can prove the result for $n$ if one assumes the result for $n+1$ and $p \nmid n$. This finishes the proof as the result for $p^\nu - 1$ is deduced from the one for $p^\nu$ (which we already have deduced from the result for the range $p^{\nu-2}$ to $p^{\nu-1}$), the result for $p^\nu - 2$ is  deduced from the one for $p^\nu - 1$ and so on. 
\end{proof}

As a corollary we see that

\begin{corollary} \label{cor1_to_main} In the basis $  \{ \ov{h}^m, \ldots, \ov{h}^{m + \mu (q-1)}   \}$, the action of $\TT_\wp$ on $\DMFmodT{k}{m}$ is given by the matrix
\[
	 \begin{bmatrix} \owp^m & 0 & \ldots &   0 \\
						\ast & \owp^m&  \ldots & \vdots \\
						 \vdots &   \vdots  &     \ddots &    0 \\
						 \ast & \ldots & \ast & \owp^m \\
						 \end{bmatrix}.
\]
In particular, 
\[
	\TT_\wp^{ss}  = \owp^m I_{1 + \mu, 1 + \mu},
\]
where $I_{1 + \mu, 1 + \mu}$ is the $(1 + \mu)$-by-$(1 + \mu)$ identity matrix.
\end{corollary}

Wilson's Theorem now shows that

\begin{corollary} \label{cor2_to_main} If $\wp_1, \ldots, \wp_{q-1} \in A$ are primes with nonzero, pairwise distinct, residues modulo $T$, then
\[
	\TT_{\wp_1} \cdots \TT_{\wp_{q-1}} \equiv -1 + \text{ nilpotent matrix } \bmod T.
\]
\end{corollary}

Corollary \ref{cor1_to_main} explicitly determines the semisimple part of $\TT_{\wp}$ modulo $T$ and we obtain the following result.

\begin{theorem} \label{hecke_algebra_mod_T}
We have
\[
\HassT{k}{m}=\F_q[{(\F_q^*)}^m],
\]
with the isomorphism being given by 
\[
	(\TT_\wp \mod{T})^{\text{ss}} \mapsto \owp^m.
\]
\end{theorem}

As mentioned in the introduction our result is in contrast with the classical case where the Hecke operators modulo $2$ are nilpotent. Corollary~\ref{cor2_to_main} also shows this contrast, since for every classical modular form with integer coefficients $f$ there exists a positive integer $i$ with the property that
\[
	\TT_{\ell_1} \cdots \TT_{\ell_i} f \equiv 0 \bmod 2,
\]
for every collection of odd primes $\ell_1, \ldots, \ell_i$. For more on this see pages 34-35 in \cite{ono2004web}. These results already suggest that the Hecke algebra action on $\DMF{k}{m}$ has a different structure than in the case of classical modular forms. We conjecture another difference in Conjecture~\ref{main_conjecture} below.

\begin{remark} \label{other_primes}
There is nothing special about $T$ as opposed to any other degree one monic irreducible. One can therefore formulate the results with obvious changes for the Hecke action modulo $(T + \theta)$, $\theta \in \F_q$, or indeed modulo
\[
	 \prod_{\theta \in \F_q} (T + \theta) = T^q - T.
\]
\end{remark}

\section{The Hecke Action on $\DMF{k}{m}$} \label{no_reduction_mod_T}

\subsection{Computations when $q$ is prime}

In this subsection we use the computation of the Hecke action modulo $T$ to conclude that certain powers of $h$ are not eigenforms under the assumption that $q$ is prime.

For the duration of this subsection we assume that $q = p$. We will keep writing~$q$, but the reader should be aware that our assumption that $q$ is prime is needed for the proofs that we present.

\begin{lemma} \label{needed}
Let $\wp$ be a prime of degree $d$, $\wp \neq T$. For $j$ in the range $1 \leq j \leq (q-1)$, the form $\ov{h}^{q^d +j}$ is not an eigenform for $\TT_\wp$.
\end{lemma}

\begin{proof}
We will show that for $j$ in the range $1 \leq j \leq (q-1)$
\[
	\TT_\wp h^{q^d + j}  \equiv j \owp^j h^{j + 1} +  \text{ higher order terms in $h$}\mod{T}
\]
by induction on $j$. 

First, consider the case $j = 1$. We want to compute the $t^2$-term in the $t$-expansion of $\TT_\wp \ov{h}^{q^d + 1}$.

Let $a_n$ be the coefficients in the $t$-expansion of $\ov{h}^{q^d + 1}$, i.e., let
\[
	\ov{h}^{q^d + 1} = \sum_{n \geq q^d + 1} a_n t^n = t^{q^d + 1} + \text{higher order terms in $t$}.
\]
Let $G_{n, \wp}$ be the $n$-th Goss polynomial of the $\wp$-torsion of the Carlitz module.

Since $t_\wp$ starts with $t^{q^d}$ (see \eqref{formula_for_ta}), formula \eqref{action_on_t_exp} shows that $\TT_\wp \ov{h}^{q^d + 1}$ is determined by 
\[
	\sum_{n \geq q^d + 1} a_n G_{n, \wp} (\wp t).
\]
According to the proof of \cite[(3.9)]{Gek}, 
\[
	\text{ord}_X G_{n, \wp} (X) \geq  \frac{n-1}{q^d} + 1.
\]
 Since $(n-1)/q^d + 1 > 2 \iff n > q^d + 1$, we see that the $t^2$-term in the $t$-expansion of $\TT_\wp h^{q^d + 1}$ is determined by $a_{q^d + 1} G_{q^d + 1, \wp} (\wp t) = G_{q^d + 1, \wp} (\wp t)$.
 
 It remains to compute $G_{q^d + 1, \wp} (X)$. We use formula \eqref{Goss_poly_mod_T}, with $n = q^d$ and $m = 1$ to get
 \[
 	G_{q^d + 1, \wp} (X) \equiv \frac{\alpha_d}{\alpha_0} X^2 + \text{higher order terms in $X$} \mod{T}
\]
But $\wp$ is monic of degree $d$ and $\alpha_0 = \owp$, therefore $G_{q^d + 1, \wp} (\wp t) \equiv \owp t^2 + \cdots \mod{T}$, and 
\[
	\TT_{\wp} h^{q^d + 1} \equiv  \owp t^2 + \text{ higher order terms in $t$ } \mod{T}.
\]
This finishes the case $j = 1$. 

To finish the proof, write
\[
\TT_\wp h^{q^d + j} \equiv \epsilon_{j+1} t^{j+1} + \text{ higher order terms in $t$} \mod{T}.
\]

Assume that
\[
	\TT_{\wp} h^{q^d + j -1} \equiv (j-1) \owp^{j-1} t^{j} + \text{ higher order terms in $t$} \mod{T},
\]
apply $\owp \Theta$ to both sides and use \eqref{differentiate_hecke_mod_T}. We get
\[
\begin{aligned}
	\TT_{\wp} \left ( \Theta h^{q^d + (j-1)} \right) & \equiv (j-1) \TT_\wp h^{q^d + j} \\
	& \equiv \owp^j j (j-1) t^{j+1} + \text{higher order terms in $t$} \mod{T}.
\end{aligned}
\]
Thus $\epsilon_{j+1} = j \owp^j$.
\end{proof}

\begin{theorem} \label{h^qd} Let $d$ be a positive integer. If $\wp$ is a monic irreducible of degree $d$, different from $T$, then $h^{q^d + j}$ is not an eigenform for $\TT_\wp$ for any $j = 1, \ldots, q-1$.
\end{theorem}

\begin{proof} If $h^{q^d + j}$ is an eigenform for $\TT_\wp$, then $\ov{h}^{q^d + j}$ is an eigenform for $\TT_\wp$. This contradicts Corollary \ref{needed} and shows that $h^{q^d + j}$ cannot be an eigenform for~$\TT_\wp$.
\end{proof}

We continue to assume that $q$ is prime. The computations above can also be used to show that $h^{q+j}$ cannot be an eigenform for $\TT_\wp$, when $\wp$ has a $T$-term. 

We will need to work with the coefficients of the $t$-expansions of several powers of~$\ov{h}$ at the same time, therefore let
\[
	\ov{h}^m = \sum_{n = m}^\infty a_n (m) t^n, \qquad \qquad a_n (m) \in \F_q.
\]

\begin{lemma} If
\[
	\ov{h}^{q+1} = \sum_{n = q+1}^\infty a_n (q+1) t^n ,
\]
then for $i \geq 2$, 
\[
	a_{q^i + 1} (q + 1) = 0.
\]
\end{lemma}

\begin{proof}
Observe that
\[
	\sum_{n = 2} a_n (2) t^n = \ov{h}^2 = -\Theta \ov{h} = - \Theta \left ( \sum_{n = 1}^\infty a_n (1) t^n \right) = \sum_{n = 1}^\infty n a_n (1) t^{n+1}.
\]
Therefore for any $s > 0$, 
\[
	a_{s q + 1} (2) = 0.
\]
Using $\ov{h}^{2 q} = (\ov{h}^2)^q$, this gives
\[
	a_{s q^2 + q} (2 q) = 0.
\]
In particular, $a_{q^{i} + q} (2 q) = 0$ for $i \geq 2$.

Consider $\ov{h}^{q + 1}$. Apply $(-1)^{q-1} \frac{1}{(q-1)!} \Theta^{q-1}$ to 
\[
	\ov{h}^{q+1} =  \sum_{n = q+1}^\infty a_n (q+1) t^n,
\]
to get
\[
	\sum_{m = 2 q}^\infty a_m (2 q) t^m = \frac{1}{(q-1)!} \sum_{n = q-1}^\infty \left ( n (n+1) \cdots (n + q - 1) \right) a_n(q+1) t^{n + q-1}.
\]
For $n = q^i + 1$ this gives
\[
	a_{q^i + q} (2 q) = \frac{1}{(q-1)!} \left ( (q^i + 1) (q^i + 2) \cdots (q^i + q - 1) \right) a_{q^i + 1} (q + 1).
\]
Since $a_{q^i + q}(2q) = 0$, ($i \geq 2$), while $\frac{1}{(q-1)!} \left ( (q^i + 1) (q^i + 2) \cdots (q^i + q - 1) \right) \not \equiv 0 \mod{T}$, this completes the proof. 
\end{proof}

\begin{theorem} \label{Hecke_action_on_h_leq2q} Let $\wp = T^d + \alpha_{d-1} T^{d-1} + \cdots + \alpha_1 T + \alpha_0$. 
For $j \in \mathbb{Z}$, $1 \leq j \leq (q-1)$, we have
\[
	\TT_\wp h^{q + j}  \equiv   \owp^j \alpha_1 h^{j+1} + \owp^{j+1} h^{q + j} \mod{T}.
\]
\end{theorem}

\begin{proof}
Let
\[
	\TT_\wp h^{q + j} \equiv \epsilon_{j+1} h^{j+1} + \epsilon_{q + j} h^{q + j} \mod{T}.
\]
Theorem \ref{main} shows that $\epsilon_{q + j} = \owp^{j+1}$, therefore we need to compute $\epsilon_{j +1}$.  We will start with the case $j = 1$ and use $\Theta$ to get the other cases. To that end, consider the $t^2$-term in the $t$-expansion of $\TT_\wp \ov{h}^{q+1}$. We know that the $t^2$-term is completely determined by
\[
	\sum_{n = q+1}^{q^d + 1} a_n (q+1) \cdot G_{n, \wp} (\wp t).
\]
But if we look at the explicit formula for $G_{n, \wp} (X)$ modulo $T$, i.e., formula \eqref{Goss_poly_mod_T}, and take $m = 1$, we can see that 
\[
	G_{q^i + 1, \wp} (X) \equiv \frac{\alpha_i}{\alpha_0} X^2 + \cdots \equiv \alpha_i \owp^{-1} X^2 + \cdots   \mod{T}, 
\]
while the other $G_{n, \wp} (X)$ in the range $q +1 \leq n \leq q^d + 1$ do not have an $X^2$-term.
Therefore 
\[
	\TT_\wp h^{q + 1} \equiv \owp^2 t^2 \sum_{i = 1}^{d} a_{q^i + 1} (q+1) \cdot \alpha_i \owp^{-1} + \text{higher order terms in $t$} \mod{T}.
\]
As $a_{q + 1} (q+1) = 1$, and $a_{q^i + 1} (q+1) = 0$ for $i \geq 2$, we see that
\[
	\TT_\wp h^{q + 1} \equiv \owp \alpha_1 t^2 + \text{higher order terms in $t$} \mod{T}.
\]
This proves the case $j = 1$. For the other cases, just apply $\wp^{j-1} \Theta^{j-1}$ to 
\[
	\TT_\wp h^{q + 1} \equiv \owp \alpha_1 h^2 + \owp^{q+1} h^{q + 1} \mod{T}.
\]
\end{proof}

Theorem \ref{Hecke_action_on_h_leq2q} shows that $\ov{h}^{q+j}$ is an eigenform modulo $T$ for $\wp$ precisely when $\alpha_1 = 0$. Therefore 

\begin{theorem} Let $\wp = T^d + \cdots + \alpha_1 T + \alpha_0$, $\alpha_i \in \F_q$, be a monic irreducible,  $\wp \neq T$. We have $\alpha_1 = 0$ if and only if $h^{q + j}$ ($j = 1, \ldots, p-1$) is an eigenform for $\TT_\wp$.
\end{theorem}

In the next subsection we present several examples which show that in the case of Drinfeld modular forms Hecke operators can have inseparable minimal polynomials.

\subsection{Hochschild Cohomology and a Non-Smoothness Conjecture}

In this subsection we recall some of the basics of Hochschild cohomology and state a non-smoothness conjecture, Conjecture \ref{main_conjecture}, about $\He{k}{m}$ for some $k \gg 0$ and some $m$. Our reference for Hochschild cohomology is \cite[Ch~9]{Weibel}.

Recall that a finite-dimensional semisimple algebra $R$ over a field $K$ is called \emph{separable} if $R_L = R \otimes L$ is semisimple for every extension $L/K$. If $M$ is an $R$-bimodule, then we let $H^\ast(R, M)$ denote the Hochschild cohomology of $R$ with coefficients in~$M$ (see \cite[Sec~9.1.1]{Weibel} for the precise definitions). The separability of a finite-dimensional algebra~$R$ turns out (see \cite[Thm~9.2.11]{Weibel}) to be equivalent to the vanishing of $H^\ast (R, M)$ for all $M$, $\ast \neq 0$. The equivalence classes of commutative $K$-algebra extensions of $R$ by $M$ are in $1-1$ correspondence with elements of $H^2 (R, M)$.  We say that a  $K$-algebra $R$ is \emph{smooth} if and only if $H^2 (R, M) = 0$ for all $R$-modules $M$ (see \cite[Pr.~9.3.4]{Weibel}). 

Now take $K = \F_q(T)$, $R = \He{k}{m}$, $M = \DMF{k}{m}$. The $K$-algebra $\He{k}{m}$ is commutative and therefore we make $\DMF{k}{m}$ into an $\He{k}{m}$-bimodule by declaring that the Hecke operators act on the right by their usual action on the left, i.e., \[ f \TT := \TT f \qquad \qquad \text{ for $f \in \DMF{k}{m}$, $\TT \in \He{k}{m}$. } \]

The $K$-algebra $\He{k}{m}$ is not separable in general as the following examples show. 

\begin{example} \label{example1} Let $q = 2$. Since the type is determined modulo $(q-1)$, we know that $m = 0$. The following table shows the first weights $k$ for which the minimal polynomial of $\TT_\wp$ is not separable for all $\wp$ of degree $\leq 5$:
\begin{center}
\begin{tabular} {| c | c | c | c | c | c | c | c| c| c| c| c | c| c|c |c| c|} \hline
 $k$ & 9 & 13 &  15 & 16 & 17 & 18 & 19 & 21 & 23 & 24 & 25 & 26 & 27 & 28 \\ \hline
 $\dim_K \DMF{k}{0}$ & 4& 5 & 6 & 6 & 6 & 7 & 7 & 8 & 8 & 9 & 9 & 9 & 10 & 10 \\ \hline
\end{tabular}
\end{center}
We note that for $k \geq 3$ $\DMF{k}{0}$ has two one-dimensional Hecke invariant subspaces (the space of Eisenstein series and the space of single-cuspidal forms), therefore the space $\DMF{9}{0}$ is actually the first space for which the Hecke action can be inseparable. 
\end{example}

\begin{example} \label{example2} If we consider modular forms of higher level, then there are examples of inseparability of the Hecke algebra even for weight $2$. Indeed \cite[(9.7.4)]{Gek96} shows that if $q = 2$, $\mathfrak{n} = T (T^2 + T + 1)$, then the Hecke action on $M_{2, 1} (\Gamma_0 (\mathfrak{n}))$ is not semisimple, therefore not separable. We wish to thank Ernst-Ulrich Gekeler for bringing our attention to this example.
\end{example}

\begin{example} \label{example3}
We also have examples of inseparability for level $T$ when $q \neq 2$. For instance, \cite[Prop.~19]{MeeLi} shows that for  $q \geq 3$ and $\wp$ a degree one prime, $\wp \neq T$, the Hecke action of $\TT_\wp$ on $M_{3, m} (\Gamma(T))$ is not diagonalizable for all types $m$. Consequently, the Hecke algebra is not semisimple and therefore not separable. We want to thank Ralf Butenuth for bringing our attention to this example.
\end{example}

Because of such examples we conjecture that:

\begin{conj} \label{main_conjecture}
Given $q$ there exist $k \gg 0$ and $m$, $0 \leq m < q$, $k \equiv 2 m \bmod (q-1)$, such that
\[
	H^2 (\He{k}{m}, \DMF{k}{m}) \neq 0.
\]
\end{conj}

In particular, the conjecture implies that $\He{k}{m}$ is not smooth for some $k \gg 0$ and some $m$.

\begin{remark} The conjecture is stronger than non-smoothness. Indeed, non-smoothness is equivalent to $H^2( \He{k}{m}, M) \neq 0$ for some $\He{k}{m}$-module $M$, while we conjecture that a natural module $\DMF{k}{m}$ afforded by our situation already detects this. 
\end{remark}

\begin{remark} Classically the self-adjointness of the Hecke operators under the Petersson inner product shows that the Hecke algebra $H_k$ is semisimple. Since $H_k$ is commutative, it is a product of fields. These fields are separable because we are in characteristic $0$. This shows that the algebra $H_k$ is \'{e}tale over $\Q$.  Since \'{e}tale implies smooth, we see that this shows another structural difference between the classical case and the case of Drinfeld modular forms.  
\end{remark}

\section{Acknowledgements}

In the preparation of this article we made extensive use of the computer packages SAGE (\cite{Sage}), PARI/GP (\cite{Pari}) and SINGULAR (\cite{Singular}), as well as the computational facilities at the University of Arizona and Texas A\&M University at Qatar. We wish to thank Matthew Johnson for his help with the PARI/GP computations. 

We want to thank Ralf Butenuth, Ernst-Ulrich Gekeler, David Goss, Matthew Johnson and Dinesh Thakur for their valuable comments and suggestions during the writing of the present work.

\bibliographystyle{amsplain}
\bibliography{references}

\end{document}